\documentclass{article}
\usepackage [utf8] {inputenc}
\usepackage{mathtools}
\usepackage{amsthm}
\usepackage{amssymb}

\newtheorem{theorem}{Theorem}
\newtheorem{lemma}{Lemma}

\DeclareMathOperator{\g}{\mathfrak{g}}
\DeclareMathOperator{\B}{\mathfrak{B}}

\title{Rota--Baxter operators on compact simple Lie groups and algebras}
\author{Saveliy V. Skresanov}
\date{}

\begin{document}
\maketitle

\begin{abstract}
	A Rota--Baxter operator on a Lie group $ G $ is a smooth map $ B : G \to G $ such that $ B(g)B(h) = B(gB(g)hB(g)^{-1}) $ for all $ g, h \in G $.
	This concept was introduced in 2021 by Guo, Lang and Sheng as a Lie group analogue of Rota--Baxter operators of weight 1 on Lie algebras.
	We show that the only Rota--Baxter operators on compact simple Lie groups are the trivial map and the inverse map.
	A similar description for Rota--Baxter operators of weight 1 on compact simple Lie algebras is provided.
\end{abstract}

\section{Introduction}

A \emph{Lie group} is a group which is also a finite-dimensional real smooth manifold,
such that group multiplication and inversion are smooth maps. We consider only real Lie groups and algebras in this paper.

In 2021 Guo, Lang and Sheng~\cite{gls} introduced the concept of a Rota--Baxter operator on a Lie group,
being motivated by some connections to the factorization theorem of Semenov--Tian--Shansky~\cite{sts}.
Namely, given a Lie group \( G \), a smooth map \( B : G \to G \) is a \textit{Rota--Baxter operator} on \( G \), if
for all \( g, h \in G \) we have
\[ B(g)B(h) = B(gB(g)hB(g)^{-1}). \]
Such operators are closely related to the Rota--Baxter operators on Lie algebras, which find applications
in the study of the Yang--Baxter equation, see~\cite{sts} and~\cite[\S 1.2]{gls} for more details,
and~\cite{bard1, bard2} for some properties of Rota--Baxter operators on abstract groups.

To make the connection between Rota--Baxter operators on Lie groups and algebras more precise,
let \( \g \) be a Lie algebra and let \( \B : \g \to \g \) be a linear operator.
We say that \( \B \) is a \emph{Rota--Baxter operator of weight~1} on \( \g \), if for all \( u, v \in \g \) we have
\[ [\B(u), \B(v)] = \B([\B(u), v] + [u, \B(v)] + [u, v]), \]
where \( [\cdot, \cdot] \) denotes the Lie bracket on~\( \g \).
If \( \g \) is the Lie algebra associated with a Lie group \( G \) and \( B : G \to G \) is a Rota--Baxter operator on \( G \),
then the tangent map at the identity \( T_1 B : \g \to \g \) is a Rota--Baxter operator of weight~1 on the Lie algebra \( \g \), see~\cite[Theorem~2.9]{gls}.

It is an interesting question which Rota--Baxter operators on the Lie algebra \( \g \) can be integrated
to the Lie group \( G \), i.e.\ whether for a given Rota--Baxter operator \( \B : \g \to \g \) there exists a Rota--Baxter operator \( B : G \to G \)
with \( \B = T_1 B \). It was shown in~\cite[Theorem~5.3]{local} that this can always be done locally (in some neighborhood of the identity in \( G \)),
and by~\cite[Corollary~2.19]{sheng}, it can be done globally if \( G \) is connected, simply connected and compact.
Moreover,~\cite[Example~4.3]{sheng} gives an example of a Lie algebra \( \g \) and an associated noncompact Lie group \( G \),
such that some Rota--Baxter operator on \( \g \) cannot be integrated to~\( G \).

The main result of this paper gives the full description of Rota--Baxter operators on compact simple Lie groups.
Recall that a Lie group is \emph{simple}, if it is connected, nonabelian and has no closed connected normal subgroups besides the whole group or the identity.
It can be easily checked that the trivial map \( B(g) = 1 \), \( g \in G \), which we denote by \( 1_G \),
and the inverse map \( B(g) = g^{-1} \), \( g \in G \), which we denote by \( (\cdot)^{-1} \), are Rota--Baxter operators on any Lie group~\( G \).
It turns out that these are the only Rota--Baxter operators on compact simple Lie groups.

\begin{theorem}\label{comp}
	The only Rota--Baxter operators on a compact simple Lie group \( G \) are \( 1_G \) and \( (\cdot)^{-1} \).
\end{theorem}

One cannot drop the condition on compactness from Theorem~\ref{comp}, since noncompact simple Lie groups have Rota--Baxter operators not of the two types above,
see the construction in~\cite[Example~2.8]{gls} using the Iwasawa decomposition.
Similarly, we cannot relax simplicity to semisimplicity, as shown by the example \( G = H \times H \) for any compact simple Lie group \( H \)
and \( B : G \to G \) defined by \( B(x, y) = x^{-1} \), see~\cite[Lemma~2.6]{gls}.

A Lie algebra is called \emph{compact} if it is a Lie algebra associated to a compact Lie group.
Recall that a simple Lie algebra is compact if and only if its Killing form is negative definite~\cite[Theorem~1.45 and Proposition~4.27]{knapp}.
For any Lie algebra \( \g \), the zero map \( 0 : \g \to \g \), defined by \( 0(u) = 0 \) for \( u \in \g \), and
the minus identity map \( -\mathrm{id} : \g \to \g \), defined by \( (-\mathrm{id})(u) = -u \) for \( u \in \g \),
are Rota--Baxter operators. It follows from Theorem~\ref{comp} and~\cite{sheng} that these are the only Rota--Baxter operators
on compact simple Lie algebras.

\begin{theorem}\label{alg}
	The only Rota--Baxter operators on a compact simple Lie algebra are \( 0 \) and \( -\mathrm{id} \).
\end{theorem}
\begin{proof}
	Let \( \g \) be a compact simple Lie algebra with a Rota--Baxter operator \( \B : \g \to \g \).
	By Lie's third theorem~\cite[Theorem~1.14.3]{dkolk}, there exists a connected simply connected Lie group \( G \) with the Lie algebra~\( \g \).
	By Weyl's theorem~\cite[Theorem~4.69]{knapp}, it is compact, and hence by~\cite[Corollary~2.19]{sheng},
	there exists a Rota--Baxter operator \( B : G \to G \) such that \( \B = T_1 B \).
	By Theorem~\ref{comp}, we have \( B = 1_G \) or \( B = (\cdot)^{-1} \), and the claim follows by computing the tangent maps.
\end{proof}

In~\cite[Example~4.1]{sheng} it is shown that any Rota--Baxter operator on the Lie algebra \( \g = \mathfrak{so}(3) \)
can be integrated to a Rota--Baxter operator on the compact simple Lie group \( G = \mathrm{SU}(2) \). In view of Theorem~\ref{alg},
this is not surprising, since the only Rota--Baxter operators on \( \g \) are \( 0 \) and \( -\mathrm{id} \)
and they correspond to operators \( 1_G \) and \( (\cdot)^{-1} \) on~\( G \).

In Section~\ref{sec1} we give the proof of Theorem~\ref{comp}, which is essentially topological in nature and relies
on deep results of Scheerer, Baum and Browder~\cite{scheer, baum} on homotopy of simple Lie groups.
In Section~\ref{sec2} we give a second proof of Theorem~\ref{alg} which does not rely on Theorem~\ref{comp}
or~\cite{sheng}, and instead uses the classification of decompositions of compact simple Lie algebras into a sum of subalgebras by Onishchik~\cite{onish}.
Although this second proof looks purely algebraic on the surface, the results of Onishchik also rely on topological properties of Lie groups.

\section{Proof of Theorem~\ref{comp}}\label{sec1}

We start with preliminary results on Lie groups and Rota--Baxter operators.
Given two Lie groups \( G \) and \( H \), we say that \( f : G \to H \) is a \emph{Lie group homomorphism} if
it is a smooth homomorphism between groups \( G \) and~\( H \).

\begin{lemma}[{\cite[Theorem~3.20]{kirillov}}]\label{lie}
	Let \( G \) and \( H \) be Lie groups with Lie algebras \( \g \) and \( \mathfrak{h} \), respectively.
	Assume \( G \) to be connected.
	If \( f, k : G \to H \) are Lie group homomorphisms and \( T_1 f = T_1 k \), then \( f = k \).
\end{lemma}

For a Lie group \( G \) with the associated Lie algebra \( \g \) we will write \( (G, \cdot) \)
and \( (\g, [\cdot, \cdot]) \) to indicate the group multiplication and the Lie bracket, when these are not clear from the context.

\begin{lemma}[{\cite[Proposition~2.13]{gls}}]\label{mul}
	Let \( (G, \cdot) \) be a Lie group with the Lie algebra \( (\g, [\cdot, \cdot]) \),
	and let \( B : G \to G \) be a Rota-Baxter operator on~\( G \) with the tangent map \( \B = T_1 B : \g \to \g \).
	Define \( * : G \times G \to G \) by \( g * h = gB(g)hB(g)^{-1} \) for all \( g, h \in G \).
	Then \( (G, *) \) is a Lie group with the Lie algebra \( (\g, [\cdot, \cdot]_{\B}) \), where
	\[ [u, v]_{\B} = [\B(u), v] + [u, \B(v)] + [u, v] \text{ for all } u, v \in \g. \]
	The map \( B \) is a Lie group homomorphism from \( (G, *) \) to \( (G, \cdot) \),
	while \( \B \) is a Lie algebra homomorphism from \( (\g, [\cdot, \cdot]_{\B}) \) to \( (\g, [\cdot, \cdot]) \).
\end{lemma}

Given a Lie group \( G \) with the associated Lie algebra \( \g \), recall that the trivial \( 1_G : G \to G \)
and inverse \( (\cdot)^{-1} : G \to G \) maps are Rota--Baxter operators on \( G \) with tangent maps equal to \( 0 : \g \to \g \) and \( -\mathrm{id} : \g \to \g \),
respectively, see~\cite[Example~2.3]{gls}. Turns out the converse is true for compact simple Lie groups.

\begin{lemma}\label{zeroinv}
	Let \( G \) be a compact simple Lie group with the associated Lie algebra~\( \g \).
	Let \( B : G \to G \) be a Rota--Baxter operator on \( G \), and let \( \B = T_1 B : \g \to \g \) be the tangent map.
	If \( \B = 0 \), then \( B = 1_G \), and if \( \B = -\mathrm{id} \), then \( B = (\cdot)^{-1} \).
\end{lemma}
\begin{proof}
	Let \( B \) be a Rota--Baxter operator with \( \B = 0 \). By Lemma~\ref{mul}, there exists another multiplication
	\( * : G \times G \to G \) such that the associated Lie algebra of the Lie group \( (G, *) \)
	is \( (\g, [\cdot, \cdot]_{\B}) \), where
	\[ [u, v]_{\B} = [\B(u), v] + [u, \B(v)] + [u, v] \text{ for all } u, v \in \g. \]
	Since \( \B(u) = \B(v) = 0 \) in our case, we yield \( [u, v]_{\B} = [u, v] \) for all \( u, v \in \g \).
	Therefore Lie algebras \( (\g, [\cdot,\cdot]) \) and \( (\g, [\cdot,\cdot]_{\B}) \) are isomorphic.

	Lie groups \( (G, \cdot) \) and \( (G, *) \) have the same underlying manifolds, in particular, they have the same homotopy type
	and are both connected and compact.
	Since \( (G, \cdot) \) is a simple Lie group, its Lie algebra \( (\g, [\cdot, \cdot]) \) is simple~\cite[Theorem~1.11.7]{dkolk}.
	Hence the algebra \( (\g, [\cdot, \cdot]_{\B}) \) is also simple, and by~\cite[Theorem~1.11.7]{dkolk} the group \( (G, *) \) is simple.
	By~\cite[Theorem~9.3]{baum}, two compact simple Lie groups of the same homotopy type are isomorphic (as Lie groups),
	hence there exists a smooth group isomorphism \( f : (G, \cdot) \to (G, *) \).

	By Proposition~\ref{mul}, the map \( B \) is a Lie group homomorphism from \( (G, *) \) to \( (G, \cdot) \).
	Hence the composition \( B \circ f : (G, \cdot) \to (G, \cdot) \) is a Lie group homomorphism from \( (G, \cdot) \) into itself.
	Consider its tangent at the identity:
	\[ T_1 (B \circ f) = T_{f(1)} B \circ T_1 f = T_1 B \circ T_1 f = 0, \]
	where \( f(1) = 1 \) since \( f \) is a homomorphism, and \( T_1 B = 0 \) by our assumption.
	The trivial map \( 1_G : (G, \cdot) \to (G, \cdot) \) also has a zero tangent, hence by Lemma~\ref{lie},
	we have \( B \circ f = 1_G \). Since \( f \) is an isomorphism, we derive \( B = 1_G \), as claimed.

	Now, if \( B \) is a Rota-Baxter operator with \( \B = -\mathrm{id} \), then by~\cite[Proposition~2.4]{gls},
	a smooth map \( \tilde{B}(g) = g^{-1} B(g^{-1}) \), \( g \in G \), is also a Rota--Baxter operator on \( G \).
	Let \( \exp : \g \to G \) denote the exponential map. For \( u \in \g \) we have
	\begin{multline*}
	(T_1 \tilde{B})(u) = \left.\frac{d}{dt}\right|_{t=0} \tilde{B}(\exp(tu)) = \left.\frac{d}{dt}\right|_{t=0} \exp(-tu)B(\exp(-tu)) =\\
		= -u + \left.\frac{d}{dt}\right|_{t=0} B(\exp(-tu)) = -u + (T_1 B)(-u) = -u + u = 0.
	\end{multline*}
	Therefore \( T_1 \tilde{B} = 0 \), so by the preceding paragraphs, \( \tilde{B}(g) = 1 \) for all \( g \in G \).
	This immediately implies \( B(g) = g^{-1} \), \( g \in G \), as claimed.
\end{proof}

We note that an analogue of Lemma~\ref{zeroinv} can be proved for connected simply connected Lie groups without assuming compactness or simplicity,
if one uses Lie's second theorem~\cite[Theorem~3.14]{kirillov} to obtain the isomorphism \( f \) in the proof.
\medskip

\noindent\emph{Proof of Theorem~\ref{comp}.}
Let \( B : G \to G \) be a Rota--Baxter operator on~\( G \). Let \( \mathfrak{g} \) be the Lie algebra of~\( G \).
By~\cite[Theorem~2.9]{gls}, the tangent map \( \mathfrak{B} : \mathfrak{g} \to \mathfrak{g} \) to \( B \) at the identity
is a Rota--Baxter operator of weight~1 on the Lie algebra \( \mathfrak{g} \). By Proposition~\ref{mul}, there is
a new smooth multiplication \( * : G \times G \to G \) such that \( G_B = (G, *) \) is a Lie group,
and a new Lie bracket \( [\cdot, \cdot]_{\mathfrak{B}} : \mathfrak{g} \times \mathfrak{g} \to \mathfrak{g} \)
such that \( \mathfrak{g}_{\mathfrak{B}} = (\mathfrak{g}, [\cdot, \cdot]_{\mathfrak{B}}) \) is a Lie algebra corresponding to the Lie group \( G_B \).
Moreover, \( \mathfrak{B} : \mathfrak{g}_{\mathfrak{B}} \to \mathfrak{g} \) is a homomorphism of Lie algebras.

If \( \mathfrak{B} \) is the zero map, then \( B(g) = 1 \) for all \( g \in G \) by Lemma~\ref{zeroinv}, and we are done.
If \( \mathfrak{B} \) is an invertible linear map, then by~\cite[Corollary~2.22]{gubarev}, \( \mathfrak{B} = -\mathrm{id} \).
Therefore \( B(g) = g^{-1} \) for all \( g \in G \) by Lemma~\ref{zeroinv}, and we are done again.
Hence we may assume that \( \mathfrak{B} \) has a nontrivial kernel, in particular, \( \mathfrak{g}_{\mathfrak{B}} \) is not a simple Lie algebra.

Note that \( G_B \) is a compact connected Lie group which is homotopy equivalent to \( G \),
since it has the same underlying manifold as~\( G \). By~\cite{scheer}, \( G_B \) and \( G \) are locally isomorphic,
in particular, their Lie algebras are isomorphic. The Lie algebra \( \mathfrak{g}_{\mathfrak{B}} \) is not simple,
while \( \g \) is simple, a contradiction. We described all possibilities for \( B \) and the proof is complete.
\qed

\section{Another proof of Theorem~\ref{alg}}\label{sec2}

We give a proof of Theorem~\ref{alg} which does not involve Rota--Baxter operators on Lie groups.
We start with the result of Onishchik on decompositions of compact simple Lie algebras.

\begin{lemma}[{\cite[Theorem~4.1 and Table~7]{onish}}]\label{decs}
	Let \( \g \) be a compact simple Lie algebra, and let \( \g_1, \g_2 \) be proper nonzero subalgebras of \( \g \) such that \( \g = \g_1 + \g_2 \).
	Then \( \g_1 \cap \g_2 \) is nonzero and at least one of the summands \( \g_i \), \( i = 1, 2 \),
	is a compact simple Lie algebra of rank smaller than the rank of \( \g \).
\end{lemma}

The lemma implies that there are no decompositions into a sum of proper subalgebras for rank~1 compact simple Lie algebras.
\medskip

\noindent
\emph{Another proof of Theorem~\ref{alg}.}
We prove the claim by induction on the rank of~\( \g \).
Let \( \B : \g \to \g \) be a Rota--Baxter operator of weight~\( 1 \) on~\( \g \). Suppose that \( \B \neq 0, -\mathrm{id} \); we want to derive a contradiction.
By~\cite[Corollary~2.22]{gubarev}, the operator \( \B \) is not invertible, hence its image \( \g_1 = \B(\g) \) is a proper nonzero subalgebra of \( \g \).
Recall that \( \B' : \g \to \g \) defined by \( \B'(u) = -u - \B(u) \) is also a Rota--Baxter operator on \( \g \), see~\cite[Lemma~2.2]{gubarev}, and
we may similarly assume that the image \( \g_2 = \B'(\g) \) is a proper nonzero subalgebra of \( \g \).

Since \( -u = \B(u) + \B'(u) \), we have \( \g = \g_1 + \g_2 \). By Lemma~\ref{decs}, the intersection \( \g_1 \cap \g_2 \) is nonzero
and, without loss of generality, we may assume that \( \g_1 \) is a compact simple Lie algebra of rank smaller than the rank of~\( \g \).

If \( \g \) has rank~1, then there are no decompositions of \( \g \) into a sum of two proper subalgebras, a contradiction.
We may assume that \( \g \) has rank at least~\( 2 \).
Clearly the restriction \( \B|_{\g_1} \) of \( \B \) on \( \g_1 \) is a Rota--Baxter operator of weight~\( 1 \) on \( \g_1 \), and by induction hypothesis \( \B|_{\g_1} \)
is either \( 0 \) or \( -\mathrm{id} \). Suppose that \( \B|_{\g_1} = 0 \). For \( u \in \g_1 \) we have \( \B'(u) = -u - \B(u) = -u \),
so \( \g_1 \leq \B'(\g) = \g_2 \). Hence, \( \g = \g_1 + \g_2 = \g_2 \), a contradiction.

Now assume that \( \B|_{\g_1} = -\mathrm{id} \). For \( u \in \g_1 \) we have \( -u = \B(u) + \B'(u) = -u + \B'(u) \), thus \( \B'(u) = 0 \).
Therefore \( \g_1 \leq \mathrm{Ker} \B' \), since \( u \in \g_1 \) was arbitrary. Hence \( \dim \g_1 \leq \dim \g - \dim \g_2 \)
and \( \dim \g_1 + \dim \g_2 \leq \dim \g \). On the other hand, since \( \g = \g_1 + \g_2 \), we have \( \dim \g \leq \dim \g_1 + \dim \g_2 \),
so \( \dim \g = \dim \g_1 + \dim \g_2 \). This implies \( \g_1 \cap \g_2 = 0 \), a contradiction. \qed
\medskip

Note that Theorem~\ref{alg} together with Lemma~\ref{zeroinv} give an alternative proof of Theorem~\ref{comp}
which does not rely on the result of Scheerer~\cite{scheer}.

\section*{Acknowledgements}

The author expresses his gratitude to A.A.~Galt, M.E.~Goncharov and V.Yu.~Gu\-barev for many stimulating discussions on the topic.

The research was carried out within the framework of the Sobolev Institute of Mathematics state contract. 

\bigskip

\noindent
\emph{Saveliy V. Skresanov}

\noindent
\emph{Sobolev Institute of Mathematics,}

\noindent
\emph{4 Acad. Koptyug avenue, 630090 Novosibirsk, Russia}

\noindent
\emph{Email address: skresan@math.nsc.ru}

\end{document}